\documentclass{amsart}

\usepackage{amsmath, amsthm, amssymb, verbatim}

\newtheorem{theorem}[subsection]{Theorem}
\newtheorem{lemma}[subsection]{Lemma}
\newtheorem{corollary}[subsection]{Corollary}
\newtheorem{prop}[subsection]{Proposition}

\theoremstyle{definition}
\newtheorem{definition}[subsubsection]{Definition}
\newtheorem{remark}[subsection]{Remark}

\newcommand{\dist}{\mathrm{dist}}
\newcommand{\Sect}{\mathrm{Sect}}
\newcommand{\hull}{\mathrm{hull}}
\newcommand{\diam}{\mathrm{diam}}
\newcommand{\spt}{\mathrm{spt}}
\newcommand{\Ric}{\mathrm{Ric}}

\newcommand{\sn}{\mathrm{sn}}

\newcommand{\bR}{\mathbb{R}}
\newcommand{\bH}{\mathbb{H}}

\title{The PPW conjecture in curved spaces}
\author{Nick Edelen}
\address{MIT Department of Mathematics, 77 Massachusetts Ave, Cambridge, MA, 02139}
\email{nedelen@mit.edu}

\begin{document}

\begin{abstract}
In Euclidean (\cite{ashbaugh-benguria}) and Hyperbolic (\cite{benguria-linde}) space, and the round hemisphere (\cite{ashbaugh-benguria-sphere}), geodesic balls maximize the gap $\lambda_2 - \lambda_1$ of Dirichlet eigenvalues, among domains with fixed $\lambda_1$.  We prove an upper bound on $\lambda_2 - \lambda_1$ for domains in manifolds with certain curvature bounds.  The inequality is sharp on geodesic balls in spaceforms.
\end{abstract}

\maketitle

\section{Introduction}

In the '90s Ashbaugh-Benguria \cite{ashbaugh-benguria} settled the following conjecture of Payne, Polya and Weinberger.

\begin{theorem}[PPW conjecture, \cite{ashbaugh-benguria}]\label{theorem:ppw-Rn}
Among all bounded domains in $\bR^n$, the round ball uniquely maximizes the ratio $\frac{\lambda_2}{\lambda_1}$ of first and second Dirichlet eigenvalues.
\end{theorem}

Given a bounded domain $\Omega \subset \bR^n$, the Dirichlet eigenvalues $\lambda_i = \lambda_i(\Omega)$ are solutions to the PDE
\begin{equation}\label{eqn:dirichlet-pde}
\Delta u + \lambda_i u = 0 \text{ in $\Omega$}, \quad u = 0 \text{ on $\partial\Omega$},
\end{equation}
where $\Delta$ denotes the usual Laplacian $\sum_k \partial_k^2$.

Physically the $\lambda_i$ correspond to harmonics in a flat drum of shape $\Omega$, so Theorem \ref{theorem:ppw-Rn} says that one can tell whether a drum is circular by listening to only the first two harmonics.  As an aside we mention that Theorem \ref{theorem:ppw-Rn} is very unstable: by gluing balls of various radii together with thin strips, one can construct domains with ratio $\lambda_2/\lambda_1$ arbitrarily close to the maximum, but which are far from being circular.

Payne-Polya-Weinberger \cite{ppw} originally bounded the ratio $\lambda_2/\lambda_1$ by $3$.  Their bound was subsequently improved by Brands \cite{brands}, de Vries \cite{devries}, then Chiti \cite{chiti}, until Ashbaugh-Benguria proved the sharp inequality, building on the work of Chiti and Talenti \cite{talenti}.  For more history and references see \cite{ashbaugh-benguria}.

If one considers the problem \eqref{eqn:dirichlet-pde} for domains in a curved space $M$, with the corresponding metric Laplacian, one is effectively considering harmonics on a drum with tension.  Benguria-Linde \cite{benguria-linde} extended the PPW conjecture to hyperbolic space.
\begin{theorem}[PPW for hyperbolic space, \cite{benguria-linde}]\label{theorem:ppw-Hn}
Among all bounded domains in $\bH^n$ with the same fixed first Dirichlet eigenvalue $\lambda_1$, the geodesic ball maximizes $\lambda_2$.
\end{theorem}

In $\bR^n$ the ratio $\lambda_2/\lambda_1$ is scale-invariant, but in other spaces the appropriate inequality requires one to normalize competitors by $\lambda_1$.  Ashbaugh-Benguria \cite{ashbaugh-benguria-sphere} also extended the PPW conjecture to the hemisphere in $S^n$.
\begin{theorem}[PPW for hemispheres, \cite{ashbaugh-benguria-sphere}]\label{theorem:ppw-Sn}
Among all bounded domains in the hemisphere of $S^n$ with the same fixed Dirichlet eigenvalue $\lambda_1$, the geodesic ball maximizes $\lambda_2$.
\end{theorem}

In this paper we seek to prove a general upper bound, in terms of geometric quantities, on the gap $\lambda_2 - \lambda_1$ for a bounded domain in a manifold $M$, and which reduces to the inequalities \ref{theorem:ppw-Rn}, \ref{theorem:ppw-Hn}, \ref{theorem:ppw-Sn} when $M$ is a spaceform.  The case of warped product manifolds has been considered by Miker \cite{miker} in her thesis, though we find her result less geometrically intuitive.

Before stating our Theorem we introduce some notation.  Given a Riemannian manifold $M^n$, we write $\Sect_M$, $\Ric_M$ for the sectional, Ricci curvatures (respecively).  Given a bounded domain $\Omega \subset M$, write $|\Omega|_M$ for the $n$-dimensional volume of $\Omega$, $|\partial \Omega|_M$ for the $(n-1)$ Hausdorff measure of $\partial \Omega$, and $\diam(\Omega)$ for the diameter of $\Omega$, each taken with respect to $M$'s Riemannian metric.

Let $N^n(k)$ be the spaceform of constant sectional curvature $k$.  Define the generalized sine function $\sn_k$ on $\bR$ by
\[
\sn_k(r) = \left\{ \begin{array}{l l} \frac{1}{\sqrt{k}} \sin (\sqrt{k} r) & k > 0 \\
r & k = 0 \\ \frac{1}{\sqrt{-k}} \sinh (\sqrt{-k} r) & k < 0 \end{array} \right. .
\]

The following isoperimetric inequality holds for any bounded domain $\Omega \subset N^n(k)$:
\begin{equation}\label{eqn:isop-ineq}
|\partial\Omega|_N \geq A_{n,k}(|\Omega|_N),
\end{equation}
with equality iff $\Omega$ is a geodesic ball (see \cite{schmidt}).

Fix (for the duration of this paper) $M^n$ to be a complete, simply-connected $n$-manifold with $\Sect_M \leq k$.  Then for some $\alpha \leq 1$, $M$ satisfies an isoperimetric inequality
\begin{equation}\label{eqn:weak-isop-ineq}
|\partial\Omega|_M \geq \alpha A_{n,k}(|\Omega|_M)
\end{equation}
for any bounded domain $\Omega$.  We assume throughout this paper that $\alpha > 0$, which is no real loss of generality as we only concern ourselves with a compact neighborhood of $\Omega$.

If $k \leq 0$ then $\Omega$ has a closed geodesic convex hull, which we write as $\hull\Omega$.  Using elementary comparison geometry one can verify that $\diam(\Omega) = \diam(\hull\Omega)$.

If $k > 0$, we impose the condition on $\Omega$ that we can find some strongly convex closed set, which we also write as $\hull\Omega$, containing $\Omega$ and satisfying the following properties
\begin{enumerate}
\item[A)] $\diam\Omega = \diam(\hull\Omega) < \min\{\frac{\pi}{2\sqrt{k}}, \text{injectivity radius of $M$}\},$
\item[B)] $|\hull\Omega|_M < |N(k)|_{N(k)}/2$.
\end{enumerate}
By strongly convex we mean that the minimizing geodesic connecting any two points in $\hull\Omega$ itself lies in $\hull\Omega$.  We require A) so that the exponential function $\exp_p$ is a diffeomorphism onto $\hull\Omega$, for any $p \in \hull\Omega$; we require B) so that we can ultimately work in the hemisphere of $N$.

We extend Theorems \ref{theorem:ppw-Rn}, \ref{theorem:ppw-Hn}, \ref{theorem:ppw-Sn} to prove the following inequality for the gap $\lambda_2 - \lambda_1$.
\begin{theorem}\label{theorem:eigenvalue-estimate}
Let $\Omega$ be a bounded domain in $M^n$.  If $k > 0$ let $\Omega$ be such that some $\hull\Omega$ exists.  Let $B_{\alpha,\Omega}$ be a geodesic ball in $N^n(k)$, normalized so that $\lambda_1(B_{\alpha,\Omega}) = \alpha^{-2}\lambda_1(\Omega)$.

If $\Ric \geq (n-1)K$ on $\hull\Omega$, then
\begin{equation}
\lambda_2(\Omega) - \lambda_1(\Omega) \leq \left( \frac{\sn_K(\diam\Omega)}{\sn_k(\diam\Omega)} \right)^{2n-2} (\lambda_2(B_{\alpha,\Omega}) - \lambda_1(B_{\alpha,\Omega})) .
\end{equation}
In particular, if $k = K$ then the constant factor is $1$, and the inequality is sharp on geodesic balls.
\end{theorem}

On spaceforms (i.e. when $k = K$) Theorem \ref{theorem:eigenvalue-estimate} reduces to the sharp estimates in \cite{ashbaugh-benguria}, \cite{benguria-linde}, \cite{ashbaugh-benguria-sphere}.  In Hadamard manifolds we have a more explicit estimate, due to the scaling of $\lambda_i$ in $\bR^n$.
\begin{corollary}\label{cor:hadamard}
Suppose $k = 0$, and $\Omega$ is a bounded domain in $M^n$ so that $\Ric_M \geq (n-1)K$ on $\hull\Omega$.  Then
\[
\frac{\lambda_2(\Omega)}{\lambda_1(\Omega)} - 1 \leq \frac{1}{\alpha^2} \left( \frac{\sinh(\sqrt{-K} \diam(\Omega))}{\sqrt{-K} \diam(\Omega)} \right)^{2n-2} \left( \frac{\lambda_2(B_1^n)}{\lambda_1(B_1^n)} - 1 \right).
\]
Here $B_1^n$ is the unit ball in $\bR^n$.
\end{corollary}

\begin{remark}
The constant factor in Theorem \ref{theorem:eigenvalue-estimate} is the ratio of areas of geodesics spheres:
\[
\lambda_2(\Omega) - \lambda_1(\Omega) \leq \left( \frac{|\partial B_{\diam\Omega}|_{N(K)}}{|\partial B_{\diam \Omega}|_{N(k)}}\right)^2 (\lambda_2(B_{\alpha,\Omega}) - \lambda_1(B_{\alpha,\Omega})) .
\]
\end{remark}

\begin{remark}
We emphasize that in many cases $\alpha$ can be explicitly computed.  If $k = 0$, then Croke \cite{croke} proved an isoperimetric relation
\[
|\partial\Omega| \geq c_n |\Omega|^{\frac{n-1}{n}} ,
\]
where $c_n$ is given by an integral formula of trigonometric functions.  If $n = 4$ then in fact $c_n$ is the Euclidean constant, and so $\alpha = 1$.

More generally, the Hadamard conjecture implies that if $k \leq 0$, then $\alpha = 1$.  The conjecture is known in the following case: $n = 2$, proved by Weil \cite{weil} (for $k = 0$), and Aubin \cite{aubin} ($k < 0$); $n = 3$, proved by Kleiner \cite{kleiner}; $n = 4$, proved by Croke \cite{croke} when $k = 0$.  Further, when $n = 4$ and $k < 0$, Kloeckner-Kuperberg \cite{kloeckner-kuperberg} proved that domains in $M$ which are appropriately "small" (in a quantitative sense) satisfy the Hadamard conjecture.  The problem is open for general $n$.

If the metric $g_M$ is $C^0$-close to $g_N$, then $\alpha$ can be written in terms of this bound.
\end{remark}


Our approach follows \cite{ashbaugh-benguria}, though subtleties arise in the presence of non-constant curvature.  We ``symmetrize'' a function defined on $M$ to a function defined on $N(k)$.  We prove a version of Chiti's theorem for this notion of symmetrization, which requires a ``sharp'' form of Faber-Krahn for manifolds satisfying a weak isoperimetric inequality \eqref{eqn:weak-isop-ineq}.  The constant term in Theorem \ref{theorem:eigenvalue-estimate} essentially results from the fact that symmetrization does not anymore preserve symmetric functions (Proposition \ref{prop:sym-of-sym}).

We remark that our choice of test functions differ from \cite{ashbaugh-benguria} even when $M = N(k) = \bR^n$.  Unlike \cite{ashbaugh-benguria}, we ultimately truncate all our functions to $\hull\Omega$, which slightly changes the symmetrizations.

We are not sure whether the diameter or Ricci curvature assumptions are necessary to obtain a gap bound like Theorem \ref{theorem:eigenvalue-estimate}, though they are necessary in our proof.  We mention that Benguria-Linde \cite{benguria-linde-ratio} showed that for geodesic balls in hyperbolic space, the ratio $\lambda_2/\lambda_1$ is strictly decreasing in the radius.

I thank my advisor Simon Brendle for his advice and encouragement, and for suggesting this problem.  I also thank Benoit Kloeckner for pointing out an error in an earlier version, and the referees for helpful comments, and for suggesting Corollary \ref{cor:hadamard}.

\section{Preliminaries}

Given $p \in M$, and vectors $v, w \in T_p M$, write $v \cdot w$ for the Riemannian inner product, and $|v| = \sqrt{v \cdot v}$ for the length.  $\exp_p : T_pM \to M$ denotes the the usual Riemannian exponential map.  If $f : M \to \bR$ is differentiable at $p$, then $\nabla f$ is the gradient vector.  We write $\omega_n$ for the volume of the Euclidean unit ball in $\bR^n$.

For the duration of this paper $N^n$ will denote $N^n(k)$.  We fix a $q \in N = N(k)$, and write $r_q(x) = \dist_N(x, q)$.  Given a function $f : M \to \bR_+$, define $\mu_f(t) = |f > t|_M$.  As usual we write $\spt f$ for the support of $f$.

\begin{definition}\label{def:sym}
Take a bounded domain $D \subset M$, and a non-negative integrable $f : D \to \bR_+$.  Define the \emph{decreasing (resp. increasing) symmetrizations}
\[
S^{D, N} f : N \to \bR_+, \quad S_{D, N} f : N \to \bR_+,
\]
by the formulae
\begin{align*}
&S^{D, N} f (x) = \mu_f^{-1} (|B_{r_q(x)}(q)|_N) , \\
&S_{D, N} f(x) = \mu_f^{-1} (\max\{|D|_M - |B_{r_q(x)}(q)|_N, 0\} ) .
\end{align*}

Let $S^N D$ be the geodesic ball in $N$ centered at $q$ satisfying $|S^N D|_N = |D|_M$.
\end{definition}

In casual terms the $S^{D, N} f$ (resp. $S_{D, N} f$) is the decreasing (resp. increasing) function of $r_q(x)$ fixed by the condition
\[
|S^{D,N} f > t|_N = |S_{D,N} f > t |_N = |f > t|_M  \quad \forall t > 0.
\]
Both $\spt S^{D,N} f$ and $\spt S_{D, N} f$ are contained in the closure of $S^N D$.

\begin{remark}\label{rem:decr-sym}
The decreasing symmetrization is actually independent of $D$, so long as $D \supset \spt f$.  In other words, if $D' \supset D$, then $S^{D, N} f \equiv S^{D', N} f$.  However in the definition of increasing symmetrization there is an ambiguity without specifying the domain of definition: if $f(x) = 0$, do we count that towards the domain of $f$ or not?
\end{remark}

\begin{prop}\label{prop:sym-preserves-norm}
For any $p \geq 1$, we have
\[
||f||_{L^p(D)} = ||S^{D,N} f||_{L^p(N)} = ||S_{D,N} f||_{L^p(N)} .
\]
\end{prop}

\begin{proof}
By Fubini's theorem, 
\begin{align*}
\int_D f^p 
&= p \int_0^\infty t^{p-1} |f > t|_M dt \\
&= p \int_0^\infty t^{p-1} |S^{D, N} f > t|_N dt \\
&= \int_{N} (S^{D,N} f)^p . 
\end{align*}
The case of $S_{D,N} f$ is verbatim.
\end{proof}

Take a $p \in M$, and define
\[
m_{D, p}(\rho) = |B_\rho(p) \cap D|_M.
\]
Similarly, write $m_{N}(\rho) = |B_\rho(q)|_N$.	

\begin{prop}\label{prop:sym-of-sym}
Suppose $f : D \to \bR_+$ is a decreasing function of $r_p(x) = \dist_M(x, p)$, then
\[
S^{D,N} f(x) = f( (m_{D, p}^{-1} \circ m_{N})(r_q(x))) .
\]
If, on the other hand, $f$ is increasing in $r_p$, then
\[
S_{D,N} f(x) = f((m_{D, p}^{-1} \circ m_{N})(r_q(x))) .
\]
\end{prop}

\begin{proof}
If $f$ is decreasing in $r_p$, then $f^{-1}(t, \infty) = B_\rho(p)\cap D$ for some $\rho = \rho(t)$, and so $\mu_f(t) = m_{D, p}(\rho(t))$.  Similarly, $f$ is increasing then $f^{-1}(t, \infty) = D \sim \overline{B_\rho(p)}$.  Now use the definition of $S^{D,N} f$, $S_{D,N} f$.
\end{proof}

\begin{prop}\label{prop:sym-ineq}
If $f, g : D \to \bR_+$, then
\[
\int_{S^N D} (S^{D,N} f)(S_{D,N} g) \leq \int_D fg \leq \int_{S^N D} (S^{D,N} f)( S^{D,N} g) .
\]
\end{prop}

\begin{proof}
By Fubini's theorem, we obtain
\begin{align*}
\int_D fg
&= \int_0^\infty \int_0^\infty |\{ f > s\} \cap \{ g > t\}|_M ds dt \\
&\leq \int_0^\infty \int_0^\infty \min\{ |f > s|_M, |g > t|_M \} ds dt \\
&= \int_0^\infty \int_0^\infty |\{S^{D,N} f > s\} \cap \{ S^{D,N} g > t\}|_N ds dt \\
&= \int_{S^N D} S^{D,N} f S^{D,N} g .
\end{align*}
The penultimate equality arises because both $S^{D,N} f$, $S^{D,N} g$ are decreasing functions of $r_q$ (i.e. the upper level-sets are balls concentric about $q$). By the same logic, since $S_{D,N} g$ is an increasing function of $r_q$, 
\begin{align*}
&\int_0^\infty \int_0^\infty |\{ f > s \} \cap \{ g > t \}|_M ds dt \\
& \geq \int_0^\infty \int_0^\infty \max\{ |f > s|_M + |g > t|_M - |D|_M, 0 \} ds dt \\
& = \int_0^\infty \int_0^\infty | \{ S^{D,N} f > s \} \cap \{ S_{D,N} g > t \}|_N ds dt . \qedhere
\end{align*}
\end{proof}

\begin{prop}\label{prop:sym-respects-powers}
For any $\beta > 0$, 
\[
S^{D,N} (f^\beta) = (S^{D,N} f)^\beta
\]
and similarly for $S_{D,N} f$.
\end{prop}

\begin{proof}
We have $\mu_{f^\beta}(t^\beta) = \mu_f(t)$, and hence $\mu_{f^\beta}^{-1} = (\mu_f^{-1})^\beta$.
\end{proof}

\section{Faber-Krahn and Chiti}

We need the following weak version of Faber-Krahn.  The inequality \eqref{eqn:faber-krahn} is a standard argument, but we find that despite any sharpness of the isoperimetric profile, we can still obtain a characterization of equality.  Recall the definition \eqref{eqn:weak-isop-ineq} of $\alpha$.
\begin{theorem}[weak Faber-Krahn]\label{theorem:faber-krahn}
If $\Omega$ is a bounded domain in $M$, then
\begin{equation}\label{eqn:faber-krahn}
\lambda_1(\Omega) \geq \alpha^2 \lambda_1(S^N\Omega) ,
\end{equation}
with equality if and only if
\[
S^{\Omega,N} u_1 \equiv v_1
\]
where $u_1$ is the first Dirichlet eigenfunction of $\Omega$, and $v_1$ the first Dirichlet eigenfunction on $S^N \Omega$, both normalized so that
\[
||u_1||_{L^2(\Omega)} = ||v_1||_{L^2(S^N\Omega)} .
\]
\end{theorem}

\begin{proof}
Write $S^N\Omega = B = B_R(q)$, and without loss of generality suppose $||u_1||_{L^2(\Omega)} = ||v_1||_{L^2(B)} = 1$, so of course $||S^{\Omega,N} u_1||_{L^2(B)} = 1$ also.  Let $\mu(t) = |u_1 > t|_M$.  For ease of notation write $A = A_{n,k}$ for the isoperimetric profile \eqref{eqn:isop-ineq} of the model space $N^n(k)$, and $\lambda_1 = \lambda_1(\Omega)$.

We have, for a.e. $t$,
\begin{align*}
-\mu'(t) 
&\geq |\partial\{u_1 > t\}|^2_M \left( \int_{\{u_1 = t\}} |\nabla u_1| \right)^{-1}\\
& \geq \alpha^2 A(|u_1 > t|_M)^2 \left( \int_{\{u_1 = t\}} |\nabla u_1| \right)^{-1} \\
&= \alpha^2 A(\mu(t))^2 \left( \int_{u_1 > t} -\Delta u_1 \right)^{-1} \\
&= \alpha^2 A(\mu(t))^2 \left( \lambda_1 \int_0^{\mu(t)} \mu^{-1}(\sigma) d\sigma \right)^{-1},
\end{align*}
and hence
\[
(\mu^{-1})'(s) \geq -\frac{\lambda_1}{\alpha^2} A^{-2}(s) \int_0^s \mu^{-1}(\sigma) d\sigma .
\]

Since $|B|_N = |\Omega|_M$, and $u_1 = 0$ on $\partial \Omega$, then $S^{\Omega,N} u_1$ has Dirichlet boundary conditions.  If $S^{\Omega,N} u_1 \not\equiv v_1$, then
\[
\lambda_1(S^N\Omega) < \int_B |\nabla S^{\Omega,N} u_1|^2 .
\]

Write $m(r) = |B_r(q)|_N$, and observe that $A(s) = m'(m^{-1}(s))$.  Since $S^{\Omega,N} u_1(r) = \mu^{-1}(m(r))$, we have
\[
|\nabla S^{\Omega,N} u_1|^2 = \left[ (\mu^{-1})'(m(r)) m'(r)\right]^2 .
\]
Therefore, we calculate
\begin{align*}
\lambda_1(S^N\Omega)
&< \int_B ((\mu^{-1})'(m (r)) m'(r))^2 \\
&= \int_0^R ((\mu^{-1})'(m (r)) m'(r))^2 m'(r) dr \\
&\leq \frac{\lambda_1}{\alpha^2} \int_0^R \frac{m'(r)^2}{A(m(r))^2} |(\mu^{-1})'|(m(r)) \int_0^{m(r)} \mu^{-1}(\sigma) d\sigma m'(r) dr \\
&= \frac{\lambda_1}{\alpha^2} \int_0^R \frac{A(m (r))^2}{A(m(r))^2} |(\mu^{-1})'|(m(r)) \int_0^{m(r)} \mu^{-1}(\sigma) d\sigma m'(r) dr \\
&\leq \frac{\lambda_1}{\alpha^2} \int_0^{|B|} ((-\mu^{-1})'(s)) \int_0^s \mu^{-1}(\sigma) d\sigma ds \\
&= \frac{\lambda_1}{\alpha^2} \int_0^{R} \mu^{-1}(m(r))^2 m'(r) dr \\
&= \frac{\lambda_1}{\alpha^2} \int_B (S^N u_1)^2 \\
&= \frac{\lambda_1}{\alpha^2} . \qedhere
\end{align*}
\end{proof}

Suppose $B_{\alpha,\Omega}$ is a ball in $N$, centered at $q$, with first eigenvalue $\lambda_1(B_{\alpha,\Omega}) = \lambda_1(\Omega) / \alpha^2$, and first eigenfunction $z$.  By the maximum principle and simplicity of $\lambda_1$, $z$ is a decreasing function of $r_q$.  By Faber-Krahn above, $\lambda_1(B_{\alpha,\Omega}) \geq \lambda_1(S^N\Omega)$, and hence $B \subset S^N\Omega$.  Further, if $B = S^N\Omega$ then necessarily $z \equiv S^N u_1$.

We obtain the following weak version of Chiti's theorem \cite{chiti}.
\begin{theorem}[weak Chiti]\label{theorem:chiti}
Let $\Omega \subset M$ be a bounded domain with first eigenvalue $\lambda_1(\Omega)$, and first eigenfunction $u_1$.  Let $B_{\alpha,\Omega} = B_R(q)$ be a ball in $N$ with first eigenvalue $\lambda_1(B_{\alpha,\Omega}) = \lambda_1(\Omega) / \alpha^2$, and first eigenfunction $z$.  Let $u_1$ and $z$ be normalized so that
\[
||u_1||_{L^2(\Omega)} = ||z||_{L^2(B_{\alpha,\Omega})}.
\]

Then there is an $r_0 \in (0, R)$ so that
\begin{align*}
&z \geq S^{\Omega,N} u_1 \text{ on $[0, r_0]$} \\
&z \leq S^{\Omega,N} u_1 \text{ on $[r_0, R]$} .
\end{align*}
\end{theorem}

\begin{proof}
Let $\mu(t) = |u_1 > t|_M$ and $\nu(t) = |z > t|_N$.  Write $\lambda_1 = \lambda_1(\Omega)$.  Recall we had
\[
(\mu^{-1})'(s) \geq -\frac{\lambda_1}{\alpha^2} A^{-2}(s) \int_0^s \mu^{-1}(\sigma) d\sigma .
\]
By repeating the proof of this with $\nu$ instead of $\mu$, we obtain
\begin{align*}
(\nu^{-1})'(s) 
&= -\lambda_1(B_{\alpha,\Omega}) A^{-2}(s) \int_0^s \nu^{-1}(\sigma) d\sigma \\
&= -\frac{\lambda_1}{\alpha^2} A^{-2}(s) \int_0^s \nu^{-1}(\sigma) d\sigma .
\end{align*}

The normalization implies $s_0 = \sup \{ s \in (0, |B|_N) : \mu^{-1}(s) \leq \nu^{-1}(s) \}$ is defined and positive.  If $s_0 = |B|_N$, then since $\nu^{-1}(|B|_N) = 0$ and $\mu^{-1}$ is decreasing, we necessarily have that $|B|_N = |\Omega|_M$.  Otherwise $u_1$ would be zero on an open set, contradicting unique continuation.  If $|B|_N = |\Omega|_M$ then by Theorem \ref{theorem:faber-krahn} $S^{\Omega,N} u_1 \equiv z$ and the Theorem is vacuous.

So we can assume $s_0 \in (0, |B|_N)$.  Clearly $\mu^{-1} \geq \nu^{-1}$ on $[s_0, |B|_N]$, and $\mu^{-1}(s_0) = \nu^{-1}(s_0)$.  We show $\mu^{-1} \leq \nu^{-1}$ on $[0, s_0]$.

Suppose, towards a contradiction, that $\beta = \sup_{[0, s_0]} \frac{\mu^{-1}}{\nu^{-1}} > 1$.  Then we calculate, for $s \in [0, s_0]$, 
\[
(\beta \nu^{-1} - \mu^{-1})'(s) \leq -\frac{\lambda_1}{\alpha^2} A^{-2}(s) \int_0^s (\beta \nu^{-1} - \mu^{-1})(\sigma) d\sigma \leq 0 .
\]
And therefore
\[
(\beta \nu^{-1} - \mu^{-1})(s) \geq (\beta \nu^{-1} - \mu^{-1})(s_0) = (\beta - 1)\nu^{-1}(s_0) > 0
\]
for any $s \in [0, s_0]$, contradicting our choice of $\beta$.  The Theorem follows by choosing $r_0$ which satisfies $|B_{r_0}(q)|_N = s_0$.
\end{proof}

\begin{corollary}\label{corollary:chiti}
If $F : S^N \Omega \to \bR_+$ is a decreasing function of $r_q$, then
\[
\int_{S^N \Omega } (S^{\Omega,N} u_1)^2 F \leq \int_{B_{\alpha,\Omega}} z^2 F
\]
with $B_{\alpha,\Omega}$, $z$ as in Theorem \ref{theorem:chiti}.  If $F$ is an increasing function of $r_q$, then
\[
\int_{S^N \Omega} (S^{\Omega,N} u_1)^2 F \geq \int_{B_{\alpha,\Omega}} z^2 F .
\]
\end{corollary}

\begin{proof}
Let $r_0$ be as in Theorem \ref{theorem:chiti}.  For $F$ decreasing, we have that
\[
(z^2 - (S^{\Omega, N} u_1)^2)(F - F(r_0)) \geq 0,
\]
with support in $S^N \Omega$.  Therefore we have
\begin{align*}
\int_{S^N \Omega} (z^2 - (S^{\Omega, N} u_1)^2) F 
&\geq F(r_0) \left( \int_B z^2 - \int_{S^N \Omega} (S^{\Omega, N} u_1)^2 \right) = 0
\end{align*}
having used Proposition \ref{prop:sym-preserves-norm}.  The case of $F$ increasing follows similarly.
\end{proof}

\section{Proof of Theorem}

Fix (for the duration of this paper) $\Omega$, $B = B_{\alpha,\Omega}$ as in Theorem \ref{theorem:chiti}, so that $\lambda_1(B) = \lambda_1(\Omega)/\alpha^2$.  Take as before $u_1$ for the first eigenfunction of $\Omega$, and $z$ the first eigenfunction of $B$.  We will sometimes abbreviate $\lambda_i = \lambda_i(\Omega)$.

If $P : \Omega \to \bR$ is any Lipschitz function such that $P u_1$ is $L^2$ orthogonal to $u_1$, then
\begin{equation}\label{eqn:min-max}
\int_\Omega |\nabla P|^2 u_1^2 \geq (\lambda_2(\Omega) - \lambda_1(\Omega)) \int_\Omega P^2 u_1^2
\end{equation}
by min-max ($Pu_1$ has the right boundary conditions) and integration by parts.  We cook up a collection of good test functions $P_i$.

Write $r_p(x) = \dist_M(p, x) = |\exp^{-1}_p(x)|$, and define $\sigma(r)$ by the condition
\[
|B_{\sigma(r)}(q)|_N = |B_r(p) \cap \hull\Omega|_M.
\]
In the notation of Proposition \ref{prop:sym-of-sym}, $\sigma(r) = (m_N^{-1} \circ m_{\hull\Omega, p})(r)$.

Let $h : \bR_+ \to \bR_+$ be a non-negative Lipschitz function with $h(0) = 0$.  For a given $p \in \hull\Omega$, define $P_p : \hull\Omega \to T_p M$ by
\[
P_p(x) = \frac{\exp_p^{-1}(x)}{r_p} h(\sigma(r_p)) .
\]

\begin{lemma}\label{lemma:base-point}
We can choose a $p \in \hull\Omega$ so that $\int_{\Omega} P_p(x) u_1^2(x) dx  = 0$.
\end{lemma}

\begin{proof}
Define the vector field
\[
X(p) = \int_\Omega P_p u_1^2.
\]
We show the integral curves of $X$ define a mapping of $\hull\Omega$ to itself.  Since $\hull\Omega$ is convex and contained in the injectivity radius, $\hull\Omega$ is topologically a ball, and therefore $X$ must have a zero by the Brouwer fixed point Theorem.

Take $q \not\in \hull\Omega$, but near enough so $\exp_q$ is a diffeomorphism on $\hull\Omega$.  Let $p \in\hull\Omega$ be the nearest point to $p$.  By convexity, the vector $\exp_{p}^{-1}(q)$ defines a supporting hyperplane for $\hull\Omega$ at $p$.  In other words, 
\[
\exp_p^{-1}(\hull\Omega) \subset \{ v : v \cdot \exp_p^{-1}(q) \leq 0\}.
\]
By definition of $P$, we deduce $X(p) \cdot \exp_p^{-1}(q) \leq 0$ also.

Let $\phi_t(p)$ be the integral curves of $X(p)$, and define the function
\[
f(q) = \left\{ \begin{array}{l l} \dist(q, \hull\Omega) & q \not\in \hull\Omega \\ 0 & \text{else} \end{array} \right. .
\]
Since $X$ is Lipschitz we have by the above reasoning that
\[
\limsup_{t \to 0_+} \frac{f(\phi_{t}(p)) - f(p)}{t} \leq C f(p),
\]
and therefore $f(\phi_t(p)) = 0$ if $f(p) = 0$.  This shows $\phi_t$ maps $\hull\Omega$ into itself.
\end{proof}

Choose an orthonormal basis $\{e_i\}$ of $T_p M$.  Define
\[
P_i(x) = e_i \cdot P_p(x) ,
\]
where we choose and fix $p$ (as a function of $h$) as in Lemma \ref{lemma:base-point}.  So $\int_\Omega P_i u_1^2 = 0$ for each $i$, and by \eqref{eqn:min-max} we have
\[
\int_\Omega (\sum_i |\nabla P_i|^2) u_1^2 \geq (\lambda_2 - \lambda_1) \int_\Omega (\sum_i P_i^2) u_1^2 = (\lambda_2 - \lambda_1) \int_\Omega h^2(\sigma(r_p)) u_1^2 .
\]

For ease of notation, in the following we will write $g \equiv h\circ \sigma$ and $r \equiv r_p$, so that $P_i(x) = e_i\cdot \exp_p^{-1}(x) g(r)/r$.  We calculate
\begin{align*}
\frac{d}{ds} |_{s=0} P_i(\exp_p(v + sw)) 
&= \frac{d}{ds}|_{s = 0} \left( e_i \cdot (v + sw) \frac{g(|v + sw|)}{|v + sw|} \right) \\
&= e_i \cdot w \frac{g(|v|)}{|v|} + \frac{ (e_i \cdot v)(v \cdot w)}{|v|} {\frac{d}{dr}}|_{r = |v|} \frac{g(r)}{r}  .
\end{align*}

Choose an orthonormal basis $E_i$ at a fixed $x = \exp_p(v)$, such that $E_1 = \frac{\partial}{\partial r}$.  Write
\[
w_j = (D\exp_p|_v)^{-1}(E_j)
\]
and since $D\exp_p$ is a radial isometry $w_1 = \frac{v}{|v|}$.  We have
\begin{align*}
&E_1 P_i = e_i \cdot v \frac{g(r)}{r^2} + e_i \cdot v \left( \frac{g(r)}{r} \right)' \\
&E_j P_i = e_i \cdot w_j \frac{g(r)}{r} \quad \text{ $j > 1$}.
\end{align*}

Therefore
\begin{align*}
\sum_i |\nabla P_i|^2 
&= \sum_i (E_1 P_i)^2 + \sum_{j > 1, i} (E_j P_i)^2 \\
&= r^2 \left[ \frac{g(r)^2}{r^4} + 2 \frac{g(r)}{r^2} \left(\frac{g(r)}{r}\right)' + {\left(\frac{g(r)}{r}\right)'}^{2} \right] + \sum_{j > 1} |w_j|^2 \frac{g(r)^2}{r^2} \\
&= g'(r)^2 + \frac{g(r)^2}{r^2} \sum_{j>1} |w_j|^2 \\
&\leq g'(r)^2 + \frac{n-1}{\sn_k^2(r)} g(r)^2
\end{align*}
having used Rauch's theorem to deduce
\[
1 = |D\exp_p|_v(w_j)| \geq \frac{\sn_k(|v|)}{|v|} |w_j| .
\]

Recalling the definition $g = h \circ \sigma$, we estimate for a.e. $r \in r_p(\Omega)$,
\begin{align*}
g'(r)^2 + \frac{n-1}{\sn_k^2 r} g(r)^2
&= h'(\sigma(r))^2 \sigma'(r)^2 + \frac{n-1}{\sn_k^2 r} h(\sigma(r))^2 \\
&\leq C_1^2 \left( h'(\sigma(r))^2 + \frac{n-1}{\sn_k^2 \sigma(r)} h(\sigma(r))^2 \right) 
\end{align*}
where
\begin{equation}\label{eqn:c_1}
C_1 = \max_{r \in r_p(\Omega)} \left\{ \sigma'(r), \frac{\sn_k (\sigma(r))}{\sn_k (r)} \right\} .
\end{equation}

We obtain
\begin{theorem}\label{theorem:balancing-estimate}
For any Lipschitz $h: \bR_+ \to \bR_+$ with $h(0) = 0$, we can choose a point $p \in \hull\Omega$ so that
\[
(\lambda_2(\Omega) - \lambda_1(\Omega)) \int_{\Omega} u_1^2 h(\sigma(r_p))^2 \leq C_1^2 \int_\Omega u_1^2 F(\sigma(r_p)) .
\]
Here $F(t) = h'(t)^2 + \frac{n-1}{\sn_k^2(t)} h(t)^2$, and $C_1$ as in \eqref{eqn:c_1}.
\end{theorem}

\begin{corollary}\label{corollary:balancing-chiti-estimate}
If $h$, $p$ are as in Theorem \ref{theorem:balancing-estimate}, and $h$ further satisfies:
\begin{align*}
(\star) \left\{ \begin{array}{l}
\text{$h(r)$ is increasing} \\
\text{$F(r)$ is decreasing}
\end{array}\right. ,
\end{align*}
then
\[
(\lambda_2(\Omega) - \lambda_1(\Omega) ) \int_B z^2 h(r_q)^2 \leq C_1^2 \int_B z^2 F(r_q) .
\]
Here $B$ and $z$ are as in Theorem \ref{theorem:chiti}.
\end{corollary}

\begin{remark}
In Corollary \ref{corollary:balancing-chiti-estimate} we have still not used the lower Ricci curvature bound.
\end{remark}

\begin{proof}
Extend $u_1$ by $0$ to be define on $\hull\Omega$, and recall that Remark \ref{rem:decr-sym} implies
\begin{equation}\label{eqn:sym-extended-u}
S^{\hull\Omega, N}u_1 \equiv S^{\Omega, N} u_1.
\end{equation}
We calculate
\begin{align*}
\int_{\Omega} u_1^2 F(\sigma(r_p)) 
&\leq \int_{S^N\hull\Omega} (S^{\hull\Omega,N} u_1)^2 (S^{\hull,N} (F \circ \sigma \circ r_p)) \\
&= \int_{S^N \Omega} (S^{\Omega, N} u_1)^2 F(r_q) \\
&\leq \int_B z^2 F(r_q) .
\end{align*}
In the first line we used Proposition \ref{prop:sym-ineq}; in the second line we used Proposition \ref{prop:sym-of-sym}, the definition of $\sigma(r)$, and \eqref{eqn:sym-extended-u}; in the third we used Corollary \ref{corollary:chiti}.

Using the same Theorems in the same order, but since $h$ is increasing, we have
\begin{align*}
\int_{\Omega} u_1^2 h(\sigma(r_p))^2 
&\geq \int_{S^N \hull\Omega} (S^{\hull\Omega, N} u_1)^2 (S_{\hull,N} (h \circ \sigma\circ r_p))^2 \\
&= \int_{S^N \Omega} (S^{\Omega, N} u_1)^2 h(r_q)^2 \\
&\geq \int_B z^2 h(r_q)^2 .
\end{align*}

Now plug these calculations into Theorem \ref{theorem:balancing-estimate}.
\end{proof}

\begin{proof}[Proof of Theorem \ref{theorem:eigenvalue-estimate}]
Recall that $B_{\alpha,\Omega} = B_R(q)$ was the geodesic ball in $N^n(k)$ with first eigenvalue $\lambda_1(B_{\alpha,\Omega}) = \lambda_1(\Omega)/\alpha^2$, and $z = z(r_q)$ was its first eigenfunction.  Let $J = J(r_q)$ be the radial component of the second Dirchlet eigenfunction of $B$ (c.f. equation 2.11 of \cite{ashbaugh-benguria}, section 3 of \cite{benguria-linde}, section 3 of \cite{ashbaugh-benguria-sphere}).

Notice that when $k > 0$, the assumption $|\hull\Omega|_M < |N|_N/2$ implies $S^N\Omega \supset B$ lies in the hemisphere.

Define
\[
h(t) = \left\{ \begin{array}{l l}
 \frac{J(t)}{z(t)} & t \in [0, R) \\
 \lim_{s \to R_-} w(s) & t \geq R \end{array} \right.
\]
Using Corollary 3.4 of \cite{ashbaugh-benguria} (if $k = 0$), Lemma 7.1 in \cite{benguria-linde} (if $k < 0$), or Theorem 4.1 in \cite{ashbaugh-benguria-sphere} (if $k > 0$), we deduce that $h(t)$ is increasing, and $F(t) = h'(t)^2 + \frac{n-1}{\sn_k^2(t)} h(t)^2$ is decreasing.

We can therefore apply Theorem \ref{corollary:balancing-chiti-estimate} to deduce
\[
(\lambda_2(\Omega) - \lambda_1(\Omega)) \leq C_1^2 (\lambda_2(B_{\alpha,\Omega}) - \lambda_1(B_{\alpha,\Omega})) ,
\]
with $C_1$ as in \eqref{eqn:c_1}.

We show that
\[
C_1 \leq \frac{|\partial B_{\diam \Omega}|_{N(K)}}{|\partial B_{\diam \Omega}|_{N(k)} }.
\]
For ease of notation write $m_\ell(r) = |B_r|_{N(\ell)}$.  All balls in $M$ are centered at $p$, and balls in $N(k)$, $N(K)$ are centered at $q$, $\tilde q$ (resp.).

Suppose $C_p$ is a geodesic cone in $M$, centered at $p$, with solid angle $\gamma n\omega_n$ in $T_pM$.  If $\Ric_M \geq (n-1)K$ on $B_r \cap C_p$, then by the Bishop-Gromov volume comparison we have
\[
|\partial B_r \cap C_p|_M \leq \gamma |\partial B_r|_{N(K)} .
\]

Conversely, choosing a linear isometry $\iota : T_pM \to T_q N(k)$, take 
\[
C'_p = (\exp^{N(k)}_q \circ \iota \circ (\exp^M_p)^{-1})(C_p)
\]
to be a geodesic cone in $N(k)$ with the same cone angle as $C_p$.  Since $\Sect_M \leq k$ we have by Hessian comparision that
\[
|B_r \cap C_p|_M \geq |B_r \cap C'_p|_{N(k)} = \gamma |B_r|_{N(k)} .
\]

Recall that $\sigma(r) = m^{-1}_k (|B_r(p) \cap \hull\Omega|_M)$.  Notice that
\[
B_r(p)\cap \hull\Omega \supset B_r(p) \cap C_p
\]
where $C_p$ is a geodesic cone at $p$ over $\partial B_r(p) \cap \hull\Omega$.  Therefore
\begin{align*}
\sigma'(r)
&= \frac{1}{m'_k (m^{-1}_k (|B_r \cap \hull\Omega|_M))} |\partial B_r \cap \hull\Omega|_M \\
&\leq \frac{1}{m'_k ( m^{-1}_k (|B_r\cap C_p|_M))} |\partial B_r \cap C_p|_M \\
&\leq \frac{1}{m'_k (m^{-1}_k (\gamma |B_r|_N))} \gamma |\partial B_r|_{N(K)} \\
&\leq \frac{|\partial B_r|_{N(K)}}{|\partial B_r|_{N(k)}}.
\end{align*}
The last inequality follows because the isoperimetric profile $A_{n,k}(s) = m'_k ( m_k^{-1} (s))$ is concave.  We elaborate.  The last inequality is equivalent to
\[
m'_k (m_k^{-1} (s)) \leq \frac{m'_k ( m^{-1}_k (\gamma s))}{\gamma}
\]
for any $\gamma \in (0, 1]$.  But the RHS is a dilation of the graph of the LHS, hence the inequality follows if the graph is concave.  We calculate
\[
(m'_k\circ m^{-1}_k)'' = \frac{(m'_k\circ m^{-1}_k) (m'''_k\circ m^{-1}_k) - (m''_k\circ m^{-1}_k)^2}{(m'_k\circ m^{-1}_k)^3} .
\]
Since
\[
(m'_k m'''_k - (m''_k)^2)(r) = -(n-1) n^2\omega_n^2 \sn_k(r)^{2n-4} \leq 0,
\]
the graph is concave (here again we use that $S^N \Omega$ lies in the hemisphere of $N(k)$, if $k > 0$).

We prove now the inequality
\[
\frac{\sn_k (\sigma(r))}{\sn_k (r)} \leq \frac{|\partial B_r|_{N(K)}}{|\partial B_r|_{N(k)}} .
\]
Since $\sigma(r) \leq m^{-1}_k (m_K ( r) )$, it suffices to prove the inequality
\[
m_K (r) \leq m_k \left[ \sn_k^{-1} \left( \frac{m'_K (r)}{m'_k (r)} \sn_k (r) \right) \right] .
\]

We therefore calculate
\begin{align*}
m_k\left[\sn_k^{-1}\left( \frac{m'_K(r)}{m'_k(r)} \sn_k(r) \right) \right] 
&= m_k\left[ \sn_k^{-1} \left( \sn_K(r) \left(\frac{\sn_K(r)}{\sn_k(r)}\right)^{n-2} \right) \right] \\
&\geq m_k\left[ \sn_k^{-1} (\sn_K(r)) \right] \\
&= n\omega_n \int_0^{\sn_k^{-1}(\sn_K(r))} \sn_k(\rho)^{n-1} d\rho \\
&= n\omega_n \int_0^r \sn_K(\rho)^{n-1} \sqrt{\frac{1-K \sn_K(\rho)^2}{1-k\sn_K(\rho)^2}} d\rho \\
&\geq n\omega_n \int_0^r \sn_K(\rho)^{n-1} d\rho \\
&= m_K(r) ,
\end{align*}
using that $\sn_k'(r)^2 = 1-k\sn_k(r)^2$.
\end{proof}

\bibliographystyle{alpha}

\end{document}